\documentclass[10pt]{amsproc}

\usepackage{graphicx}

\usepackage{tikz}
\usetikzlibrary{calc}

\usepackage{amsmath}
\usepackage{amsthm}
\usepackage{amsfonts}
\usepackage{amssymb}

\copyrightinfo{}{}

\newtheorem{theorem}{Theorem}[section]
\newtheorem{Lemma}[theorem]{Lemma}
\newtheorem{prep}[theorem]{Preposition}

\theoremstyle{definition}

\theoremstyle{remark}
\newtheorem{remark}[theorem]{Remark}

\begin{document}

\title{Closing curves by rearranging arcs}


\author{Leonardo Alese}
\address{TU Graz, Department of Mathematics, Institute of Geometry }
\curraddr{}
\email{alese@tugraz.at}


\date{\today}

\begin{abstract}
In this paper we show how, under surprisingly weak assumptions, one can split a planar curve into three arcs and rearrange them (matching tangent directions) to obtain a closed curve. We also generalize this construction to curves split into $k$ arcs and comment what can be achieved by rearranging arcs for a curve in higher dimensions. Proofs involve only tools from elementary topology, and the paper is mostly self-contained.
\end{abstract}

\maketitle

\section{Introduction.} 
In this paper we study the problem of splitting a given planar curve into arcs and rearrange them (matching tangent directions) in order to make the curve closed. The interest of our result lies in the counterintuitive nature of the statement and in the simplicity of the proof. Using an argument very similar to that involved in the topological proof of the fundamental theorem of algebra, we will show that a one time differentiable curve with total turning angle a non-zero integer multiple of $2\pi$ can always be split into three arcs that are rearrangeable to a closed curve. 

The operation of joining arcs of curves matching frames at junction points has been considered in various settings, mainly  with the goal of constructing closed curves with certain properties. In \cite{MR2408486} multiple copies of the same planar curve are joined one after another; if the arc-length integral of the curvature is a rational non-integer multiple of $2\pi$ then gluing finitely many copies of the curve will eventually close up the construction to the starting point. In \cite{MR1647587} and \cite{MR2493076} arcs of helices resp. so-called Salkowski curves are joined to obtain a family of closed space curves of constant curvature which are curvature-continuous ($C^2$) resp. $C^3$. In \cite{MR2327737} all types of knots and links are realized as $C^2$ curves of constant curvature by joining arcs of helices.

As for an outline of the contents, in \S \ref{not} we set the notation and recall an elementary topology lemma. In \S \ref{main} we prove that, under very natural assumptions, a $C^1$ planar curve can be split into $3$ arcs that can be rearranged to obtain a closed $C^1$ curve. In \S \ref{more} this construction is extended to permutations of any number of arcs. In \S \ref{highersec} generalizations to higher dimensions are discussed and some possible directions for future work are pointed out.

\section{Notation.} \label{not}

We start by settling the language about some natural topological objects. For any $a,b \in \mathbb{R}$ with $a\leq b$, we call \textit{path} a continuous function $w(t)$ over $[a,b]$ to $\mathbb{R}^2$ and \textit{loop} a path $l$ such that $l(a)=l(b)$. A \textit{contraction} of a loop $l:[a,b] \rightarrow \mathbb{R}^2$ to a point $Q$ is a continuous family $H(h,t)$ defined for $h,t\in[a,b]$ with $h\leq t$ such that $H(a,t)=l(t)$, $H(h,h)= H(h,b)$ for all $h\in[a,b]$ and $H(b,b)=Q$: in particular, each path $t \mapsto H(h,t)$ is a loop defined over $[h,b]$. A loop $l$ is said to be \textit{contractible} in a subset $A$ of $\mathbb{R}^2$ if there exists a contraction of $l$ whose image is contained in $A$. 

A path $w(t)$ over $[a,b]$ that does not contain $P \in \mathbb{R}^2$ can be expressed in \textit{polar coordinates} $(\rho(t),\phi(t))$ with respect to $P$ as $w(t)=P+\rho(t)(\cos \phi(t),\sin \phi(t))$. Since $w$ is continuous, we will always assume $\rho(t)$ and $\phi(t)$ are continuous as well. We call \textit{winding number} of $w$ with respect to $P$ the value $\phi(b)-\phi(a)$. The next is a result from elementary homotopy theory and the proof given here is the usual one contained in topology textbooks; for more on homotopy theory see for example the first chapter of \cite{hatcher}.

\begin{Lemma} \label{contr}
Let $l(t)$ be a loop on $[a,b]$ whose image does not contain $P \in \mathbb{R}^2$. Let $(\rho(t),\phi(t))$ be polar coordinates with respect to $P$ and $\phi(b)-\phi(a)=2k\pi$, $k \in \mathbb{Z}$. If $k \neq 0$, i.e. $l$ has winding number with respect to $P$ different from $0$, then $l$ is not contractible in $\mathbb{R}^2 \setminus \{P\}$.
\end{Lemma}
\begin{proof}
Assume for a contradiction that a contraction $H(h,t)$ exists. For each loop $H(h,t)$ we consider polar coordinates $(\rho_{h}(t),\phi_{h}(t))$ and the integer $k_{h}$ that satisfies $\phi_{h}(b)=\phi_{h}(a)+2k_{h}\pi$. Since a contraction is a continuous function, $\phi_h(t)$ can be chosen to be continuous also in $h$ and such that $\phi_a(t)=\phi(t)$ for all $t$ (we continuously extend the polar coordinates we already have for $l$). This entails that $k_h$ is also continuous and therefore constant as a function whose image is contained in $\mathbb{Z}$. This implies $k_h=k\neq 0$ for all $h$. But this is a contradiction since the image of $H(b,t)$ is a single point and therefore $k_b=0$.
\end{proof} 

We formalize now what is meant by \textit{joining arcs of curves matching frames at junction points}. To do that we introduce the concept of \textit{framed} curve $(\gamma,\mathcal{F})(s)$ of $\mathbb{R}^n$, a pair consisting of a $C^1$ curve $\gamma(s)$ of $\mathbb{R}^n$ with constant speed $c$, namely a differentiable function from some interval $[a,b]$ to $\mathbb{R}^n$ such that $\|\gamma'\| \equiv c>0$, and a positive orthonormal basis $\mathcal{F}(s)=\{f_1(s),f_2(s),...,f_n(s)\}$ of $\mathbb{R}^n$, continuous in $s$ and such that $f_1(s)=\gamma'(s)/c$. An example of a frame for $n=3$ and $\gamma \in C^2$ with everywhere non-zero curvature is the Frenet-Serret frame, obtained by taking $f_2$ as $\gamma''/\|\gamma''\|$ and $f_3$ as the vector product of $f_1$ and $f_2$.

If $(\gamma_1,\mathcal{F}_1)(s_1)$ and $(\gamma_2,\mathcal{F}_2)(s_2)$ are two framed curves parametrized over $[a_1,b_1]$ and $[a_2,b_2]$ with the same constant speed $c$, denoting with $T_{a_2,b_1}$ the rigid motion of $\mathbb{R}^n$ that shifts the point $\gamma_2(a_2)$ to $\gamma_1(b_1)$ and rotates the frame $\mathcal{F}_2(a_2)$ to $\mathcal{F}_1(b_1)$, we define the \textit{concatenation} of the two curves, for  $s \in [0,b_1-a_1+b_2-a_2]$, as
\[
  \gamma_1 * \gamma_2 (s):=
  \begin{cases}
  \gamma_1(s+a_1), &  s \leq b_1-a_1 \\
        T_{a_2,b_1}\gamma_2(\bar{s}(s)), & s > b_1-a_1,
  \end{cases}
\]
where $\bar{s}(s)=s-(b_1-a_1)+a_2$ is the reparametrization achieving $\bar{s}(b_1-a_1)=a_2$. 
Operation $*$ rigidly glues $\gamma_2$ to the end point of $\gamma_1$ by matching frames. It is easy to see that this operation is associative.

\section{Two-cut theorem.} \label{main} In this section let $\gamma(s)$ be a $C^1$ planar curve with constant speed $c$, parametrized over $[0,1]$ and framed with $\mathcal{F}(s)=\{f_1(s),f_2(s)\}$ where $f_2(s)$ is obtained by rotating $f_1(s)=\gamma'(s)/c$ counter-clockwise by $\frac{\pi}{2}$. For convenience we assume without loss of generality $\gamma(0)=(0,0)$ and $f_1(0),f_2(0)$ aligned with the axes of the coordinate system. For any choice of \textit{cuts} $c_1,c_2$ with $0 \leq c_1\leq c_2 \leq 1$, we split the curve into three arcs $\gamma_1$, $\gamma_2$ and $\gamma_3$ respectively parametrized over $[0,c_1]$,$[c_1,c_2]$ and $[c_2,1]$. We define the \textit{rearranged} curve $r_{(c_1,c_2)}:[0,1] \rightarrow \mathbb{R}^2$ as 
 \[
  r_{(c_1,c_2)}:=\gamma_1 * \gamma_3 * \gamma_2.
  \]
Figure \ref{ex} visualizes this construction for different cuts $c_1,c_2$. Intui\-tively this is the curve obtained by swapping the middle arc between parameters $c_1$ and $c_2$ and the tail of the curve between $c_2$ and $1$. The rearrangement is well defined also if one or more arcs $\gamma_i$ degenerate to a point since they still inherit from $\gamma$ the information of a tangent direction.

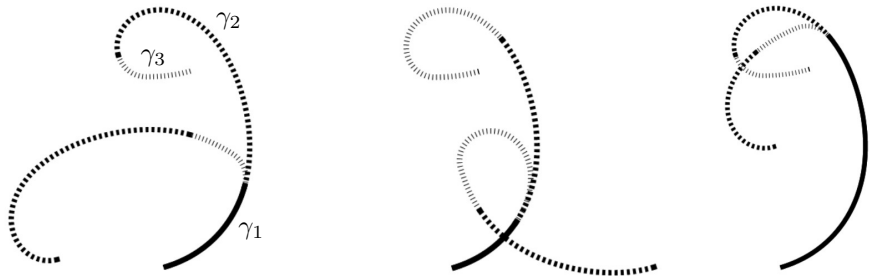
\begin{figure}
\begin{tikzpicture}
    \node[anchor=south west,inner sep=0] at (0,0) {\scalebox{0.92}{\includegraphics[width=\textwidth]{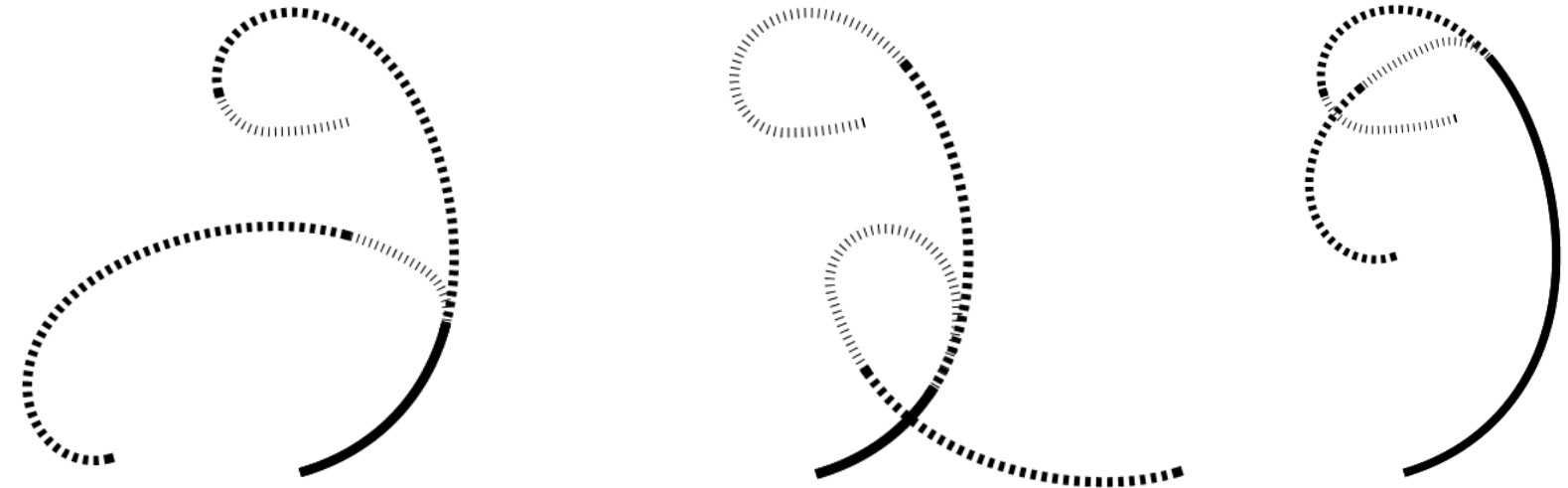}}};
    \node at (3.4,0.6) {$\gamma_1$};
	\node at (3.1,3.4) {$\gamma_2$};
	\node at (2.1,2.9) {$\gamma_3$}; 
\end{tikzpicture}
\caption{Rearrangements of a curve for different cuts. Arcs $\gamma_2$ and $\gamma_3$ are swapped and tangent directions matched.} \label{ex}
\end{figure}

In Theorem \ref{thm1} and Lemma \ref{uglyhyp} we will continuously move the two cuts, while tracking the end point of the rearranged curve, defining this way a family of loops whose properties with respect to contractibility imply the existence of cuts such that $\gamma_1 * \gamma_3 * \gamma_2$ is a closed curve. Two main observations will be needed to follow such a construction and it is worth stressing them before moving to the proofs.
\begin{itemize}
\item For any $c_1 \in [0,1]$ it holds $r_{(c_1,c_1)}(1)=r_{(c_1,1)}(1)=\gamma(1)$, i.e. $c_2\mapsto r_{(c_1,c_2)}(1)$ is a loop defined over $[c_1,1]$.
\item If $c_1=0$, the rearrangement just swaps the only two arcs in which the curve has been split. Under the hypothesis of Theorem \ref{thm1}, $\|r_{(0,c_2)}(1)-r_{(0,c_2)}(0)\|$ will not depend on the choice of $c_2$. 
\end{itemize}
In the following we consider a function $\theta(s)$ such that $\gamma'(s)=c\,(\cos\theta(s),\sin\theta(s))$ and call it \textit{turning angle} function for $\gamma$. We also call $\theta(1)-\theta(0)$ the \textit{total turning angle} of $\gamma$. Besides, we denote with $R_\theta$ the counter-clockwise rotation of angle $\theta$ and center in the origin. 
 
\begin{theorem} [Two-cut theorem] \label{thm1}
Let $\gamma(s)$ be a $C^1$ constant speed planar curve over $[0,1]$ and $\theta(s)$ a turning angle function for $\gamma$. If $\theta(1)-\theta(0)=m 2 \pi$ with $ 0 \neq  m \in \mathbb{Z} $ then there exist cuts $c_1,c_2$ such that the rearranged curve $r_{(c_1,c_2)}$ is a closed $C^1$ curve.
\end{theorem}
\begin{proof}
Let us consider the loop $e(t):=r_{(0,t)}(1)$. Explicitly,
\begin{align*}
e(t)&=R_{\theta(1)-\theta(t)}(\gamma(t)-\gamma(0))+R_{\theta(0)-\theta(t)}(\gamma(1)-\gamma(t))+\gamma(0) \\
    &=R_{\theta(1)-\theta(t)}(\gamma(t))+R_{-\theta(t)}(\gamma(1)-\gamma(t)) \\
    &= R_{-\theta(t)}\big(R_{\theta(1)}(\gamma(t))-\gamma(t)+\gamma(1)\big) = R_{-\theta(t)}\big(\gamma(1)\big),
\end{align*}
where we have used the equaliy $R_{\theta(1)}=R_{\theta(0)}$, which follows directly from the hypothesis on the total turning angle, and the assumptions $\gamma(0)=(0,0)$, $\gamma'(0)=(c,0)$. This equation provides polar coordinates $(\|\gamma(1)\|,-\theta(t))$ for $e(t)$ whose image turns out to be a circle of radius $\|\gamma(1)\|$ centered in the origin. By hypothesis we have $\theta(1)-\theta(0)=m 2 \pi$ with $ m \neq 0$ and therefore Lemma \ref{contr} guarantees $e$ is not contractible in $\mathbb{R}^2\setminus\{(0,0)\}=\mathbb{R}^2\setminus \{\gamma(0)\}$. Varying $h\in[0,1]$ we obtain a continuous family of loops $r_{(h,t)}(1)$, $t\in [h,1]$, which is in fact a contraction of $e(t)$ to $\gamma(1)$, an operation not possible in $\mathbb{R}^2 \setminus \{(0,0)\}$. Hence it must be a contraction in $\mathbb{R}^2$ whose image contains $(0,0)$, which means there exist $(h,t)$ such that $r_{(h,t)}(1)=(0,0)$, starting point of the curve. Since the total turning angle did not change and it is still an integer multiple of $2\pi$, tangents at the beginning and at the end of the curve match.
\end{proof} 

\begin{figure}
\begin{tikzpicture}
    \node[anchor=south west,inner sep=0] at (0,0) {\scalebox{1}{\includegraphics[width=\textwidth]{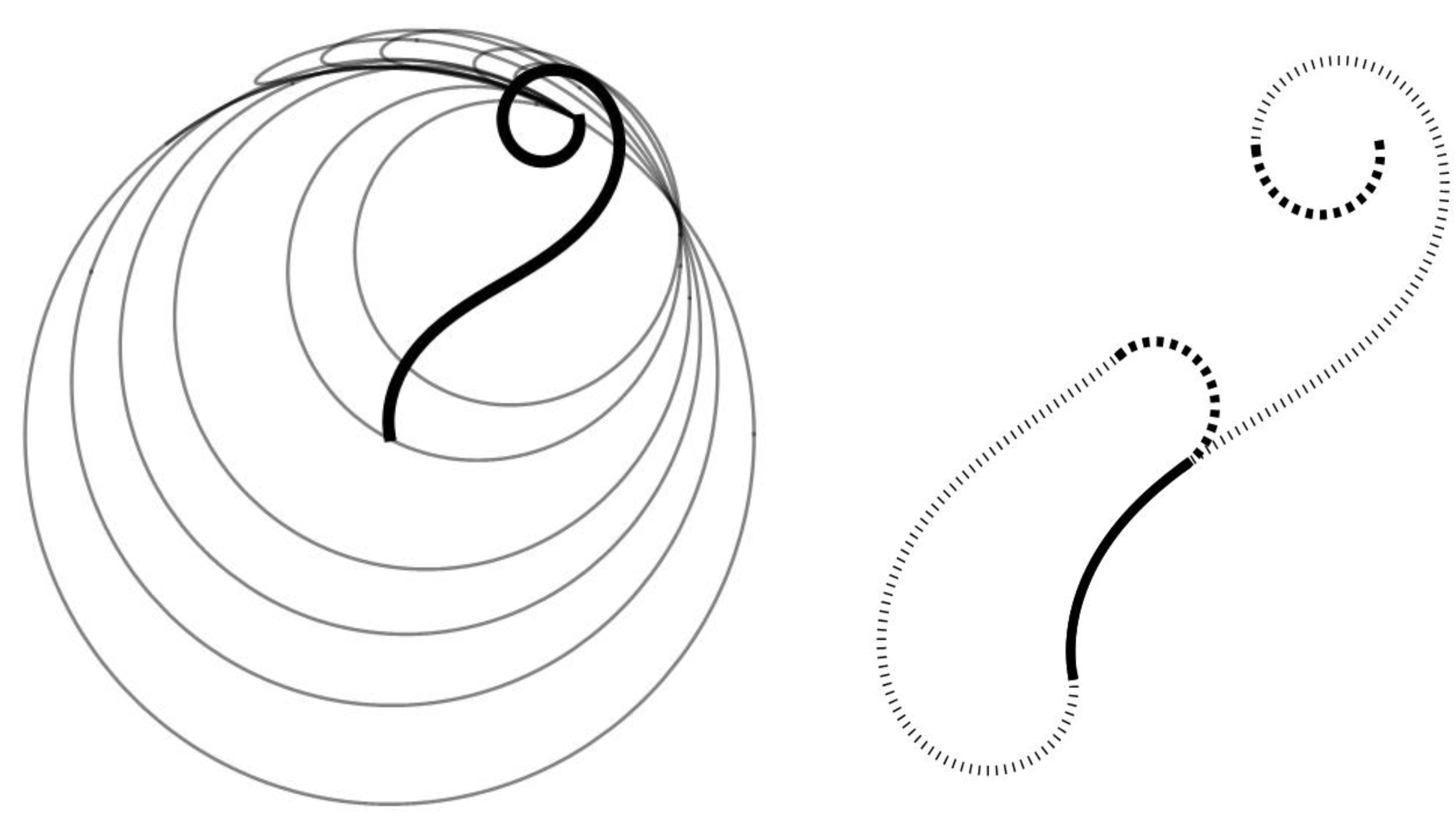}}};
   
\node at (5.8,0.2) {$e=r_{(0,t)}(1)$};
\node at (4.35,2.86) {$r_{(h,t)}(1)$};
\node at (4.95,4.6) {$\gamma$};

\node at (9.85,2) {$\gamma_1$};
\node at (12.35,4) {$\gamma_2$};
\node at (11.54,4.94) {$\gamma_3$};
\node at (8.85,4.09) {$r_{(c_1,c_2)}$};

\node at (10.8,2.9) {$\gamma(c_1)$};
\node at (10.4,5.8) {$\gamma(c_2)$};
\end{tikzpicture}
\caption{For any choice of the first cut $h$, tracking in $t$ the endpoint of the rearranged curve $r_{(h,t)}$ provides a loop. If the total turning angle of $\gamma$ is a non-zero multiple of $2\pi$ loops in this family start as a circle and contract to $\gamma(1)$, therefore passing through the origin and guaranteeing the rearrangeability to a closed curve.} 
\end{figure}
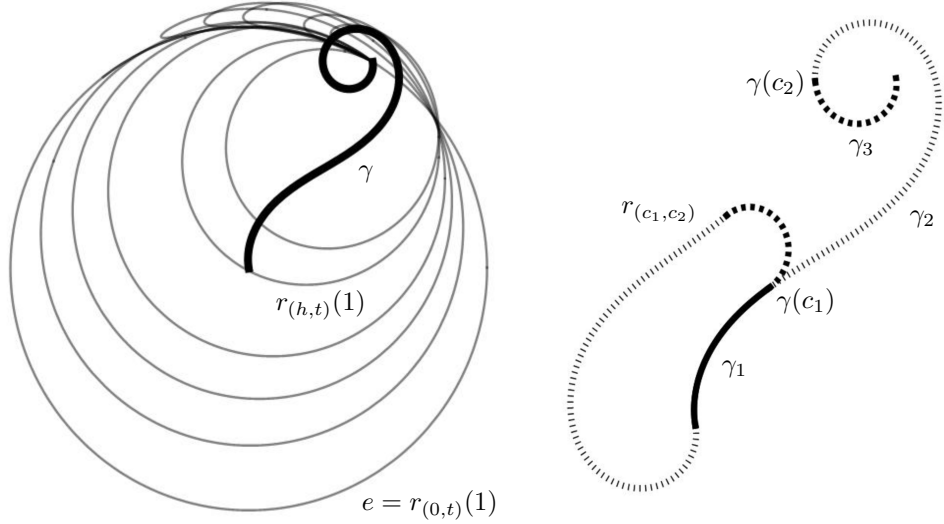

\begin{figure}
\begin{tikzpicture}
    \node[anchor=south west,inner sep=0] at (0,0) {\scalebox{1}{\includegraphics[width=\textwidth]{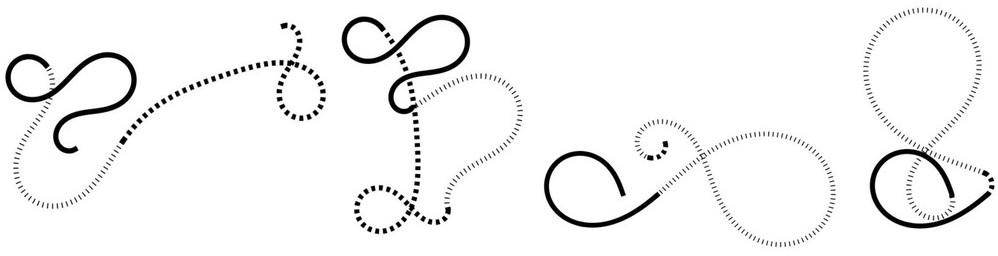}}};
\end{tikzpicture}
\caption{More closed rearrangements for curves whose total turning angle is a non-zero integer multiple of $2\pi$.} \label{star}
\end{figure}

Examples of rearrangements are given in Figure \ref{star}. As mentioned in the introduction, the proof of Theorem \ref{thm1} is similar to the topological proof of the fundamental theorem of algebra. Given a polynomial $p(z)=z^n+a_{n-1}z^{n-1}+...+a_1z+a_0$ with $n>0$, for $\rho\in \mathbb{R}$ large enough the loop $p(\rho e^{it})$, $t \in [0,2\pi]$ winds around the origin $n$ times. If now $\rho$ continuously decreases to $0$ such a loop contracts to $a_0$, which entails the existence of $\rho,t$ such that $\rho e^{it}$ is a root of $p$. A detailed proof can be found in \cite{hatcher}. 

We conclude the section with a lemma, which, using the same techinques of the proof of Theorem \ref{thm1}, generalizes the previous result at the price of a (much) less expressive hypothesis. 

\begin{Lemma} \label{uglyhyp} Let $\gamma(s)$ be a $C^1$ constant speed planar curve over $[0,1]$ and $\theta(s)$ a turning angle function for $\gamma$. If $|\theta(1)-\theta(0)| \geq 2 \pi$ and 
\[
\|\gamma(1)\| \geq  \sqrt{2\big(1-\cos\theta(1)\big)}\max_{s\in[0,1]} \|\gamma(s)\|,
\]
then there exist cuts $c_1,c_2$ such that $r_{(c_1,c_2)}$ is closed, i.e., $r_{(c_1,c_2)}(0)=r_{(c_1,c_2)}(1)$ (not necessarily smooth at the end point).
\end{Lemma}
\begin{proof}
Again we look at $e(t)=R_{-\theta(t)}\big(R_{\theta(1)}(\gamma(t))-\gamma(t)+\gamma(1)\big)$.
After computing $\|R_{\theta(1)}(\gamma(t))-\gamma(t)\|=\sqrt{2(1-\cos\theta(1))}\|\gamma(t)\|$ our hypothesis entails that $R_{\theta(1)}(\gamma(t))-\gamma(t)+\gamma(1)$ is contained in a ball centered in $\gamma(1)$ which does not contain the origin. If $(\rho(t),\phi(t))$ are polar coordinates for $e(t)$ with respect to the origin this implies $\phi(1)-\phi(0)\in(\theta(0)-\theta(1)-\frac{\pi}{2},\theta(0)-\theta(1)+\frac{\pi}{2})$. Since $e$ is a loop and $|\theta(1)-\theta(0)| \geq 2 \pi$ we get $\phi(1)-\phi(0)=2m_e\pi$ with $0\neq m_e \in \mathbb{Z}$ and Lemma \ref{contr} guarantees again that $e$ is not contractible in $\mathbb{R}\setminus\{\gamma(0)\}$. Varying $h\in[0,1]$ we define again a continuous family of loops $r_{(h,t)}(1)$, $t\in [h,1]$, which contracts to $\gamma(1)$, and we conclude as in the proof of Theorem \ref{thm1}.
\end{proof}

Note that the argument of Theorem \ref{thm1} and Lemma \ref{uglyhyp} still works if we want to reach, as the end point of the rearranged curve, any point ``inside'' the loop $e(t)$, that is any point with respect to which $e$ has winding number different from $0$. Moreover, the winding number itself provides a lower bound on the number of possible different rearrangements.  

\begin{remark}
The conditions we used on the turning angle are sufficient but not necessary for rearrangeability. It is easy to find examples of curves with total turning angle $0$, whose associated loop $e$ is not surjective onto the circle in which it is contained, but that, in spite of the failure of the contraction argument, can still be rearranged to closed curves. On the other hand, surjectivity of $\gamma'/c$ onto $S^1$ is not sufficient to guarantee that a curve is rearrangeable to a closed one; if arcs of circles and line segments are arranged as in Figure \ref{surj}, the end point $r_{(h,t)}(1)$ does not coincide with the origin for any choice of cuts.
\end{remark}

\begin{figure}
\begin{tikzpicture}
    \node[anchor=south west,inner sep=0] at (0,0) {\scalebox{1}{\includegraphics[width=\textwidth]{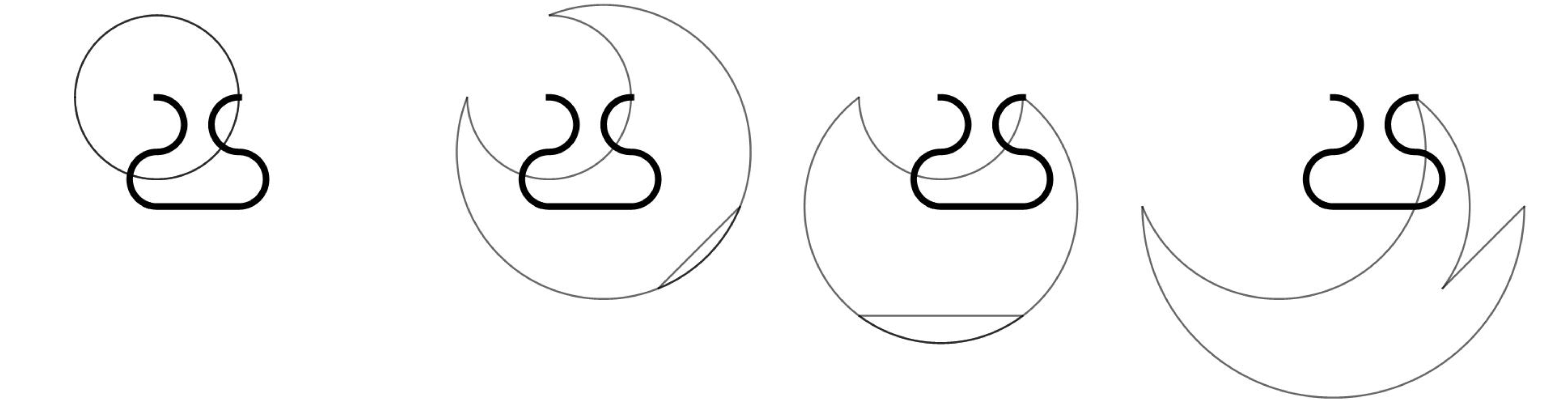}}};
    \node at (5.4,0.6) {$r_{(h,t)}(1)$};
\end{tikzpicture}
\caption{In bold a curve whose tangent is surjective onto $S^1$ but no 3 arcs can be rearranged to obtain a closed curve. The four images show loops in the family $r_{(h,t)}(1)$ for different values of $h$. Such a family opens up without passing through the origin.} \label{surj}
\end{figure}

\section{More permutations.} \label{more}
We break now the curve into $k \geq 2$ arcs by chosing cuts in the set 
$$
D_k:= \{ (c_1,c_2,...,c_{k-1}) \in [0,1]^{k-1}, \; 0 \leq c_1 \leq c_2 \leq ... \leq c_{k-1} \leq 1 \}.
$$
In the following, for notation convenience we will refer a few times also to the $0$th and $k$-th component of a string of cuts $C\in D_k$, which we define by setting $c_0:=0$ and $c_k:=1$. Given $C \in D_k$, we split a framed curve $(\gamma,\mathcal{F})$ over $[0,1]$ into $\gamma_1,\gamma_2,...,\gamma_k$ respectively defined over $[0,c_1],[c_1,c_2],..., \allowbreak [c_{k-1},1]$. For any element $\sigma$ of the permutation group $S_k$ of the indices $\{1,2,...,k\}$, we define
\[
r_{\sigma,C}:= \gamma_{\sigma(1)}*\gamma_{\sigma(2)}*...*\gamma_{\sigma(k)}.
\]
In order to compare more easily rearranged curves from different permutations, in the following we will implicitely assume that $r_{\sigma,C}$ is also shifted and rotated such that its starting point and frame coincide with the origin and the axes of the coordinate system. We also define
$$
e_\sigma: D_k \rightarrow \mathbb{R}^n, \;\;\;\;\;\; e_\sigma(C):= r_{\sigma,C}(1),
$$
which is the (continuous) map from the space of admissible cuts to the end point of the rearranged curve.
We call a curve $(k,j)$\textit{-rearrangeable with respect to a permution} $\sigma \in S_k$ if there exist cuts $C\in D_k$ such that $r_{\sigma,C}$ is a $C^j$ closed curve. Theorem \ref{thm1} and Lemma \ref{uglyhyp} from \S \ref{main} give conditions that guarantee a curve is $(3,1)$-rearrangeable resp. $(3,0)$-rearrangeable with respect to the permutation $(23)$ of $S_3$. Note that asking for higher regularity, i.e. $j>1$, greatly restricts the set of admissible cuts.

We now want to move towards a full characterization of the permutations $\sigma \in S_k$ with respect to which a planar curve can be rearranged. In particular, we will see that this characterization is the same as that of permutations $\sigma \in S_k$ with respect to which a curve can be \textit{properly} rearranged, meaning with this that no arc of the rearranged closed curve $r_{\sigma,C}$ degenerates to a point (in other words in a proper rearrangement we ask for $C$ to be contained in the interior of the set of admissible cut $D_k$, denoted with $\text{int}(D_k)$ in the following). A major role will be played by the subgroup of cyclic shifts of $S_k$,
$$
Z_k:=\{ z_h\}_{h\in\{0,1,...,k-1\}} \subseteq S_k \;\; \text{with} \;\; z_h(i)= \begin{cases}
    i+h, & i \leq k-h \\
    i+h-k , & i > k-h.  \\
  \end{cases}
$$
By the end of the section we will have proven the following theorem. 
\begin{theorem} \label{char}
Let $\gamma(s)$ be a $C^1$ constant speed non-closed planar curve over $[0,1]$, whose turning angle function $\theta(s)$ satisfies $\theta(1)-\theta(0)=m 2 \pi$ with $ 0 \neq  m \in \mathbb{Z} $. If $\sigma \in S_k$ with $3 \leq k$, then there exist cuts $C\in \text{int}(D_k) $ such that the rearranged curve $r_{\sigma,C}$ is closed $C^1$ if and only if $\sigma \in S_k \setminus Z_k$, i.e. $\sigma$ is no cyclic shift.
\end{theorem}

The next two lemmas explain how $Z_k$ is relevant to our purposes. For the rest of the section, as in the hypothesis of Theorem \ref{char}, we will always assume $\gamma(s)$ is a $C^1$ constant speed non-closed planar curve defined over $[0,1]$, whose turning angle function $\theta(s)$ satisfies $\theta(1)-\theta(0)=m 2 \pi$ with $ 0 \neq  m \in \mathbb{Z} $. 
\begin{Lemma} \label{notrearr}
Let $\gamma$ be a curve as described in the previous paragraph. If $z_h \in Z_k$ then there exist no cuts $C \in D_k$ such that $r_{z_h,C}$ is closed. 
\end{Lemma} 
\begin{proof}
For any cuts, $r_{z_h,C}$ is the same curve as the one obtained by swapping just two arcs partitioning the curve. More precisely, if $C=(c_1,c_2,...,c_{k-1})$, it holds $r_{z_h,C}=r_{(12),\bar{C}}$ with $\bar{C}=(c_{h})$, and we have already observed in \S \ref{main} that the distance from the origin of the end point of such a curve is independent from the choice of the cut and different from $0$.
\end{proof}

\begin{Lemma} \label{cyclic}
Let $\gamma$ be a curve as described above. If $\sigma \in S_k$ and $z_h \in Z_k $, then $\gamma$ is $(k,1)$-rearrangeable with respect to $\sigma$ if and only if $\gamma$ is $(k,1)$-rearrangeable with respect to the composition $\sigma \cdot z_h$ of these two permutations.
\end{Lemma}
\begin{proof}
As already pointed out, $z_h$ cyclically shifts all elements by the same integer $h$. Once the curve is rearranged to a closed one, it remains closed if the sequence of arcs is cyclically shifted and the same cuts $C \in D_k$ are used. 
\end{proof}

Observing that $(123),(132)\in Z_3$ and that $(13)=(23)\cdot(123)$, $(12)=(23)\cdot(132)$ we conclude the following characterization of $(3,1)$-rearrangeability.
\begin{prep} \label{sum}
Let $\gamma(s)$ be a $C^1$ constant speed non-closed planar curve over $[0,1]$, whose turning angle function $\theta(s)$ satisfies $\theta(1)-\theta(0)=m 2 \pi$ with $ 0 \neq  m \in \mathbb{Z} $. Let also $\sigma$ be a permutation of $\{1,2,3\}$, then $\gamma$ is $(3,1)$-rearrangeable with respect to $\sigma$ if and only if $\sigma$ is a transposition.
\end{prep}

At this point we can make clear how we want to tackle the proof of Theorem \ref{char}. The plan is to first drop the properness constraint and show that some arcs can be collapsed to reduce the problem to a permutation on $S_3$. After that we will conclude by observing that the topological argument we used for $S_3$ is robust to perturbations and can be adapted to cuts where the degnerate arcs are inflated a little bit to guarantee properness. 

When two cuts in $D_k$ coincide and an arc degenerates we can relabel the indices, inducing a permutation on $S_{k-1}$. For $i\leq k$, $\sigma \in S_k$, we define $F_i(\sigma)\in S_{k-1}$ by
\[
  F_i(\sigma)(j) =
  \begin{cases}
    \sigma(j), & j<i,\sigma(j)<\sigma(i) \\
    \sigma(j)-1, & j<i,\sigma(j)>\sigma(i) \\
    \sigma(j+1), & j\geq i,\sigma(j+1)<\sigma(i) \\
    \sigma(j+1)-1, & j\geq i,\sigma(j+1)>\sigma(i),
  \end{cases}
\]
where $j\in \{1,2,...,k-1\}$. This definition by cases might be not particularly expressive, while Fig. \ref{relabel} is the better tool to understand the combinatorial meaning of this relabelling. $F_i$'s are not group homomorphisms as shown for example in $S_5$ by taking $i=2$, $\sigma_1=(124)(35),\sigma_2=(134)(25)$ for which $F_2(\sigma_2)\cdot F_2(\sigma_1) \neq F_2(\sigma_2 \cdot \sigma_1)$.
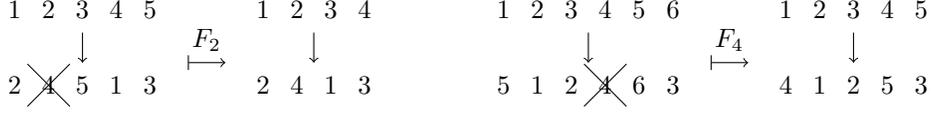
\begin{figure}
\begin{tikzpicture}
\coordinate (A) at (1,2);
\coordinate (GAPx) at (0.45,0);
\coordinate (GAPY) at (0,-1);
\coordinate (GAPxIMG) at (1.5,0);
\coordinate (GAPxPERM) at (2.5,0);
\coordinate (CROSSbl) at (-1.41/5,-1.41/5);
\coordinate (CROSSbr) at (1.41/5,-1.41/5);
\coordinate (CROSStl) at (-1.41/5,1.41/5);
\coordinate (CROSStr) at (1.41/5,1.41/5);
\coordinate (TOTSHIFT) at (6.5,0);

\node at (A) {$1$};
\node at ($(A)+(GAPx)$) {$2$};
\node at ($(A)+2*(GAPx)$) {$3$};
\node at ($(A)+3*(GAPx)$) {$4$};
\node at ($(A)+4*(GAPx)$) {$5$};

\draw[->] ($(A)+2*(GAPx)+(0,-0.30)$) -- ($(A)+2*(GAPx)+(0,-0.70)$);

\node at ($(A)+(GAPY)$) {$2$};
\node at ($(A)+(GAPx)+(GAPY)$) {$4$};

\draw[-] ($(A)+(GAPx)+(GAPY)+(CROSSbl)$) -- ($(A)+(GAPx)+(GAPY)+(CROSStr)$);
\draw[-] ($(A)+(GAPx)+(GAPY)+(CROSSbr)$) -- ($(A)+(GAPx)+(GAPY)+(CROSStl)$);

\node at ($(A)+2*(GAPx)+(GAPY)$) {$5$};
\node at ($(A)+3*(GAPx)+(GAPY)$) {$1$};
\node at ($(A)+4*(GAPx)+(GAPY)$) {$3$};

\draw[|->] ($(A)+4*(GAPx)+1/2*(GAPxIMG)-(0.25,0)-(0,0.70)$) -- ($(A)+4*(GAPx)+(0.25,0)+1/2*(GAPxIMG)-(0,0.70)$);
\node at ($(A)+4*(GAPx)+1/2*(GAPxIMG)-(0,0.40)$) {$F_2$};

\node at ($(A)+4*(GAPx)+(GAPxIMG)$) {$1$};
\node at ($(A)+4*(GAPx)+(GAPx)+(GAPxIMG)$) {$2$};
\node at ($(A)+4*(GAPx)+2*(GAPx)+(GAPxIMG)$) {$3$};
\node at ($(A)+4*(GAPx)+3*(GAPx)+(GAPxIMG)$) {$4$};

\draw[->] ($(A)+3/2*(GAPx)+(0,-0.30)+4*(GAPx)+(GAPxIMG)$) -- ($(A)+3/2*(GAPx)+(0,-0.70)+4*(GAPx)+(GAPxIMG)$);

\node at ($(A)+4*(GAPx)+(GAPY)+(GAPxIMG)$) {$2$};
\node at ($(A)+4*(GAPx)+(GAPx)+(GAPY)+(GAPxIMG)$) {$4$};
\node at ($(A)+4*(GAPx)+2*(GAPx)+(GAPY)+(GAPxIMG)$) {$1$};
\node at ($(A)+4*(GAPx)+3*(GAPx)+(GAPY)+(GAPxIMG)$) {$3$};


\node at ($(A)+(TOTSHIFT)$) {$1$};
\node at ($(A)+(GAPx)+(TOTSHIFT)$) {$2$};
\node at ($(A)+2*(GAPx)+(TOTSHIFT)$) {$3$};
\node at ($(A)+3*(GAPx)+(TOTSHIFT)$) {$4$};
\node at ($(A)+4*(GAPx)+(TOTSHIFT)$) {$5$};
\node at ($(A)+5*(GAPx)+(TOTSHIFT)$) {$6$};

\draw[->] ($(A)+5/2*(GAPx)+(0,-0.30)+(TOTSHIFT)$) -- ($(A)+5/2*(GAPx)+(0,-0.70)+(TOTSHIFT)$);

\node at ($(A)+(GAPY)+(TOTSHIFT)$) {$5$};
\node at ($(A)+(GAPx)+(GAPY)+(TOTSHIFT)$) {$1$};
\node at ($(A)+2*(GAPx)+(GAPY)+(TOTSHIFT)$) {$2$};
\node at ($(A)+3*(GAPx)+(GAPY)+(TOTSHIFT)$) {$4$};

\draw[-] ($(A)+3*(GAPx)+(GAPY)+(CROSSbl)+(TOTSHIFT)$) -- ($(A)+3*(GAPx)+(GAPY)+(CROSStr)+(TOTSHIFT)$);
\draw[-] ($(A)+3*(GAPx)+(GAPY)+(CROSSbr)+(TOTSHIFT)$) -- ($(A)+3*(GAPx)+(GAPY)+(CROSStl)+(TOTSHIFT)$);

\node at ($(A)+4*(GAPx)+(GAPY)+(TOTSHIFT)$) {$6$};
\node at ($(A)+5*(GAPx)+(GAPY)+(TOTSHIFT)$) {$3$};

\draw[|->] ($(A)+5*(GAPx)+1/2*(GAPxIMG)-(0.25,0)-(0,0.70)+(TOTSHIFT)$) -- ($(A)+5*(GAPx)+(0.25,0)+1/2*(GAPxIMG)-(0,0.70)+(TOTSHIFT)$);
\node at ($(A)+5*(GAPx)+1/2*(GAPxIMG)-(0,0.40)+(TOTSHIFT)$) {$F_4$};

\node at ($(A)+5*(GAPx)+(GAPxIMG)+(TOTSHIFT)$) {$1$};
\node at ($(A)+5*(GAPx)+(GAPx)+(GAPxIMG)+(TOTSHIFT)$) {$2$};
\node at ($(A)+5*(GAPx)+2*(GAPx)+(GAPxIMG)+(TOTSHIFT)$) {$3$};
\node at ($(A)+5*(GAPx)+3*(GAPx)+(GAPxIMG)+(TOTSHIFT)$) {$4$};
\node at ($(A)+5*(GAPx)+4*(GAPx)+(GAPxIMG)+(TOTSHIFT)$) {$5$};

\draw[->] ($(A)+2*(GAPx)+(0,-0.30)+5*(GAPx)+(GAPxIMG)+(TOTSHIFT)$) -- ($(A)+2*(GAPx)+(0,-0.70)+5*(GAPx)+(GAPxIMG)+(TOTSHIFT)$);

\node at ($(A)+5*(GAPx)+(GAPY)+(GAPxIMG)+(TOTSHIFT)$) {$4$};
\node at ($(A)+5*(GAPx)+(GAPx)+(GAPY)+(GAPxIMG)+(TOTSHIFT)$) {$1$};
\node at ($(A)+5*(GAPx)+2*(GAPx)+(GAPY)+(GAPxIMG)+(TOTSHIFT)$) {$2$};
\node at ($(A)+5*(GAPx)+3*(GAPx)+(GAPY)+(GAPxIMG)+(TOTSHIFT)$) {$5$};
\node at ($(A)+5*(GAPx)+4*(GAPx)+(GAPY)+(GAPxIMG)+(TOTSHIFT)$) {$3$};
\end{tikzpicture}
\caption{When the arc in position $i$ of the rearranged curve degenerates, $k-1$ arcs are left and a permutation on $S_{k-1}$ is induced.} \label{relabel}
\end{figure}

We now need to prove a combinatorial lemma, which will be the core of the inductive construction used in Proposition \ref{boundary} to prove the characterization of $(k,1)$-rearrangeability, when cuts are allowed (actually forced) to degenerate. 

\begin{Lemma} \label{ktokminus1}
For $\sigma \in S_k \setminus Z_k$ with $k \geq 4$ and $\sigma(1)=1$, there exists $i\in\{1,2,...,k\}$ such that $F_i(\sigma)\in S_{k-1} \setminus Z_{k-1}$.
\end{Lemma}
\begin{proof}
Let us consider the smallest $r \in \{1,2,...,k\}$ such that $\sigma(r) \neq r$, which exists since $\sigma$ is not the identical permutation. We set $i$ to the following values, distinguishing $3$ cases: 
\[
 i :=
  \left\{\begin{array}{llllr}
    1, & r>2 & & F_i(\sigma)(r-1)=\sigma(r)-1\neq r-1 &\\
    \sigma^{-1}(k), & r=2, \sigma(2)\neq k & \Rightarrow & F_i(\sigma)(2)=\sigma(2) \neq 2 &\\
    3, & r=2, \sigma(2)= k, & & F_i(\sigma)(2)=\sigma(2)-1=k-1>2. &
  \end{array}\right.
\]
Since for all three cases $F_i(\sigma)(1)=1$, this concludes the proof.
\end{proof}

\begin{prep} \label{boundary}
Let $\gamma(s)$ be a $C^1$ constant speed non-closed planar curve over $[0,1]$, whose turning angle function $\theta(s)$ satisfies $\theta(1)-\theta(0)=m 2 \pi$ with $ 0 \neq  m \in \mathbb{Z} $.  If $\sigma \in S_k$ with $4 \leq k$, then there exist cuts $C\in \partial(D_k) $ such that the rearranged curve $r_{\sigma,C}$ is closed $C^1$ if and only if $\sigma \in S_k \setminus Z_k$.
\end{prep}
\begin{proof}
Lemma \ref{notrearr} rules out permutations in $Z_k$ from the picture. By Lemma \ref{cyclic} proving rearrangeability with respect to $\sigma \cdot z_h$ for $z_h \in Z_k$ implies also rearrangeability with respect to $\sigma$. We can therefore assume, by possibly applying some cyclic shift, that $\sigma(1)=1$. By Lemma \ref{ktokminus1} we can find $i$ such that by taking cuts on $\partial D_k$ with $c_{\sigma(i)-1}=c_{\sigma(i)}$ a permutation in $S_{k-1} \setminus Z_{k-1}$ is induced. The statement follows by induction, once we observe that the base $k=4$ is guaranteed by Proposition \ref{sum} after one last contraction.
\end{proof}

We can finally prove Theorem \ref{char} by discussing the robustness of the argument we used in \S \ref{main}.
\begin{proof} [Proof of Theorem \ref{char}]
The case $k=3$ is implied by Proposition \ref{sum} after observing that cuts must be in the interior of $D_3$ since the curve we want to rearrange is not closed. If $k>3$, by Proposition \ref{boundary} we can find $q_1<q_2<q_3 \in \{0,1,2,...,k-1\}$ such that there exist cuts $\bar{C}=(\bar{c}_1,\bar{c}_2,...,\bar{c}_{k-1})$, satisfying $\bar{c}_j<\bar{c}_{j+1}$ for $j\in \{q_1,q_2,q_3\}$ and $\bar{c}_j=\bar{c}_{j+1}$ otherwise, and such that $r_{\sigma,\bar{C}}$ is closed $C^1$ (recall that we agreed on $\bar{c}_0=0$ and $\bar{c}_k=1$). We want now to inflate degenerate arcs, making them proper again, without undermining the  contraction argument we used for $S_3$. For $(l_1,l_2) \in D_3$ we let $ \delta(l_1,l_2) = \frac{1}{k-2}\min\{l_1,l_2-l_1,1-l_2\}$ and define the \textit{inflated} cuts $I[l_1,l_2]=(c_1,...,c_{k-1})\in D_k$ as
$$
c_j=\begin{cases}
j\delta(l_1,l_2), & j \leq q_1 \\
l_1+(j-(q_1+1))\delta(l_1,l_2), & q_1 < j \leq q_2 \\
l_2+(j-(q_2+1))\delta(l_1,l_2), & q_2 < j \leq q_3 \\
1 - (k-j)\delta(l_1,l_2), & j > q_3.
\end{cases}
$$

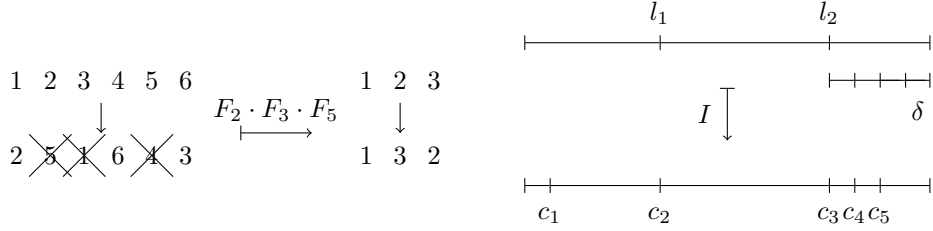
\begin{figure}
\begin{tikzpicture}
\coordinate (A) at (1,2);
\coordinate (GAPx) at (0.45,0);
\coordinate (GAPY) at (0,-1);
\coordinate (GAPxIMG) at (1.5,0);
\coordinate (GAPxPERM) at (2.5,0);
\coordinate (CROSSbl) at (-1.41/5,-1.41/5);
\coordinate (CROSSbr) at (1.41/5,-1.41/5);
\coordinate (CROSStl) at (-1.41/5,1.41/5);
\coordinate (CROSStr) at (1.41/5,1.41/5);
\coordinate (TOTSHIFT) at (6.5,0);

\node at (A) {$1$};
\node at ($(A)+(GAPx)$) {$2$};
\node at ($(A)+2*(GAPx)$) {$3$};
\node at ($(A)+3*(GAPx)$) {$4$};
\node at ($(A)+4*(GAPx)$) {$5$};
\node at ($(A)+5*(GAPx)$) {$6$};

\draw[->] ($(A)+2.5*(GAPx)+(0,-0.30)$) -- ($(A)+2.5*(GAPx)+(0,-0.70)$);

\node at ($(A)+(GAPY)$) {$2$};
\node at ($(A)+(GAPx)+(GAPY)$) {$5$};

\draw[-] ($(A)+(GAPx)+(GAPY)+(CROSSbl)$) -- ($(A)+(GAPx)+(GAPY)+(CROSStr)$);
\draw[-] ($(A)+(GAPx)+(GAPY)+(CROSSbr)$) -- ($(A)+(GAPx)+(GAPY)+(CROSStl)$);

\node at ($(A)+2*(GAPx)+(GAPY)$) {$1$};

\draw[-] ($(A)+2*(GAPx)+(GAPY)+(CROSSbl)$) -- ($(A)+2*(GAPx)+(GAPY)+(CROSStr)$);
\draw[-] ($(A)+2*(GAPx)+(GAPY)+(CROSSbr)$) -- ($(A)+2*(GAPx)+(GAPY)+(CROSStl)$);

\node at ($(A)+3*(GAPx)+(GAPY)$) {$6$};
\node at ($(A)+4*(GAPx)+(GAPY)$) {$4$};

\draw[-] ($(A)+4*(GAPx)+(GAPY)+(CROSSbl)$) -- ($(A)+4*(GAPx)+(GAPY)+(CROSStr)$);
\draw[-] ($(A)+4*(GAPx)+(GAPY)+(CROSSbr)$) -- ($(A)+4*(GAPx)+(GAPY)+(CROSStl)$);

\node at ($(A)+5*(GAPx)+(GAPY)$) {$3$};

\draw[|->] ($(A)+5.5*(GAPx)+1/2*(GAPxIMG)-(0.25,0)-(0,0.70)$) -- ($(A)+6.5*(GAPx)+(0.25,0)+1/2*(GAPxIMG)-(0,0.70)$);
\node at ($(A)+6*(GAPx)+1/2*(GAPxIMG)-(0,0.40)$) {$F_2 \cdot F_3 \cdot F_5$};

\node at ($(A)+7*(GAPx)+(GAPxIMG)$) {$1$};
\node at ($(A)+7*(GAPx)+(GAPx)+(GAPxIMG)$) {$2$};
\node at ($(A)+7*(GAPx)+2*(GAPx)+(GAPxIMG)$) {$3$};

\draw[->] ($(A)+4*(GAPx)+(0,-0.30)+4*(GAPx)+(GAPxIMG)$) -- ($(A)+4*(GAPx)+(0,-0.70)+4*(GAPx)+(GAPxIMG)$);

\node at ($(A)+7*(GAPx)+(GAPY)+(GAPxIMG)$) {$1$};
\node at ($(A)+7*(GAPx)+(GAPx)+(GAPY)+(GAPxIMG)$) {$3$};
\node at ($(A)+7*(GAPx)+2*(GAPx)+(GAPY)+(GAPxIMG)$) {$2$};


\coordinate (ANCHORINFL) at ($(A)+(0,0.5)$);
\coordinate (SHIFTDELTA) at (0,-0.5);

\draw[|-] ($(ANCHORINFL)+15*(GAPx)$) -- ($(ANCHORINFL)+19*(GAPx)$);
\node at ($(ANCHORINFL)+19*(GAPx)+(0,0.4)$) {$l_1$};
\draw[|-] ($(ANCHORINFL)+19*(GAPx)$) -- ($(ANCHORINFL)+24*(GAPx)$);
\node at ($(ANCHORINFL)+24*(GAPx)+(0,0.4)$) {$l_2$};
\draw[|-|] ($(ANCHORINFL)+24*(GAPx)$) -- ($(ANCHORINFL)+27*(GAPx)$);

\draw[|-] ($(ANCHORINFL)+24*(GAPx)+(SHIFTDELTA)$) -- ($(ANCHORINFL)+24*(GAPx)+3/4*(GAPx)+(SHIFTDELTA)$);
\draw[|-] ($(ANCHORINFL)+3/4*(GAPx)+24*(GAPx)+(SHIFTDELTA)$) -- ($(ANCHORINFL)+6/4*(GAPx)+24*(GAPx)+3/5*(GAPx)+(SHIFTDELTA)$);
\draw[|-] ($(ANCHORINFL)+6/4*(GAPx)+24*(GAPx)+(SHIFTDELTA)$) -- ($(ANCHORINFL)+9/4*(GAPx)+24*(GAPx)+3/5*(GAPx)+(SHIFTDELTA)$);
\draw[|-|] ($(ANCHORINFL)+9/4*(GAPx)+24*(GAPx)+(SHIFTDELTA)$) -- ($(ANCHORINFL)+27*(GAPx)+(SHIFTDELTA)$);
\node at ($(ANCHORINFL)+21/8*(GAPx)+24*(GAPx)+(SHIFTDELTA)+(0,-0.4)$) {$\delta$};

\draw[|->] ($(ANCHORINFL)+21*(GAPx)+(0,-0.6)$) -- ($(ANCHORINFL)+21*(GAPx)+(0,-1.3)$);
\node at ($(ANCHORINFL)+21*(GAPx)+(0,-0.95)+(-0.3,0)$) {$I$};

\draw[|-] ($(ANCHORINFL)+15*(GAPx)+(0,-1.9)$) -- ($(ANCHORINFL)+15*(GAPx)+3/4*(GAPx)+(0,-1.9)$);
\node at ($(ANCHORINFL)+15*(GAPx)+3/4*(GAPx)+(0,-1.9)+(0,-0.4)$) {$c_1$};
\draw[|-] ($(ANCHORINFL)+3/4*(GAPx)+15*(GAPx)+(0,-1.9)$) -- ($(ANCHORINFL)+19*(GAPx)+(0,-1.9)$);
\node at ($(ANCHORINFL)+19*(GAPx)+(0,-1.9)+(0,-0.4)$) {$c_2$};
\draw[|-] ($(ANCHORINFL)+19*(GAPx)+(0,-1.9)$) -- ($(ANCHORINFL)+24*(GAPx)+(0,-1.9)$);
\node at ($(ANCHORINFL)+24*(GAPx)+(0,-1.9)+(0,-0.4)$) {$c_3$};
\draw[|-] ($(ANCHORINFL)+24*(GAPx)+(0,-1.9)$) -- ($(ANCHORINFL)+24*(GAPx)+3/4*(GAPx)+(0,-1.9)$);
\node at ($(ANCHORINFL)+24*(GAPx)+3/4*(GAPx)+(0,-1.9)+(0,-0.4)$) {$c_4$};
\draw[|-] ($(ANCHORINFL)+3/4*(GAPx)+24*(GAPx)+(0,-1.9)$) -- ($(ANCHORINFL)+24*(GAPx)+6/4*(GAPx)+(0,-1.9)$);
\node at ($(ANCHORINFL)+24*(GAPx)+6/4*(GAPx)+(0,-1.9)+(0,-0.4)$) {$c_5$};
\draw[|-|] ($(ANCHORINFL)+6/4*(GAPx)+24*(GAPx)+(0,-1.9)$) -- ($(ANCHORINFL)+27*(GAPx)+(0,-1.9)$);

\end{tikzpicture}
\caption{Indices are repeatedly contracted from a permution in $S_6$ to obtain a permutation in $S_3$. Such indices determine how degenerate cuts are inflated and turned into proper ones.} \label{contrex}
\end{figure}

Figure \ref{contrex} visualizes the inflation operation. Note that if $(l_1,l_2)\in\text{int}(D_3)$ then $I[l_1,l_2]\in \text{int}(D_k)$. Further, no inflation happens if $(l_1,l_2)\in\partial(D_3)$. What we are doing here is to push the interior of the triangular face $c_j=c_{j+1}$ for $j\notin \{q_1,q_2,q_3 \}$, which can be thought as a copy of $D_3$, towards the interior of $D_k$. Because of Proposition \ref{cyclic} we can assume the permutation induced on $S_3$ is the one swapping arc $2$ and $3$. We proceed as in \S \ref{main}, by considering the loop $e_\sigma(I[0,t])$ defined on $[0,1]$. Since no inflation is happening on $\partial D_3$, this is exactly the same loop contained in a circle and with winding number $m2\pi$ we described in the proof of Theorem \ref{thm1}, which we know being contained and not contractible in $\mathbb{R}^2 \setminus \{ (0,0) \} $. For $h\in[0,1]$ the family of loops $e_\sigma(I[h,t])$ defined on $[h,1]$ is a contraction to a point different from $(0,0)$, which implies the existence of $(h,t)\in \text{int}(D_3)$ such that $e_\sigma(I[h,t])=(0,0)$.

\end{proof}

\begin{figure}
\begin{tikzpicture}
    \node[anchor=south west,inner sep=0] at (0,0) {\scalebox{0.92}{\includegraphics[width=\textwidth]{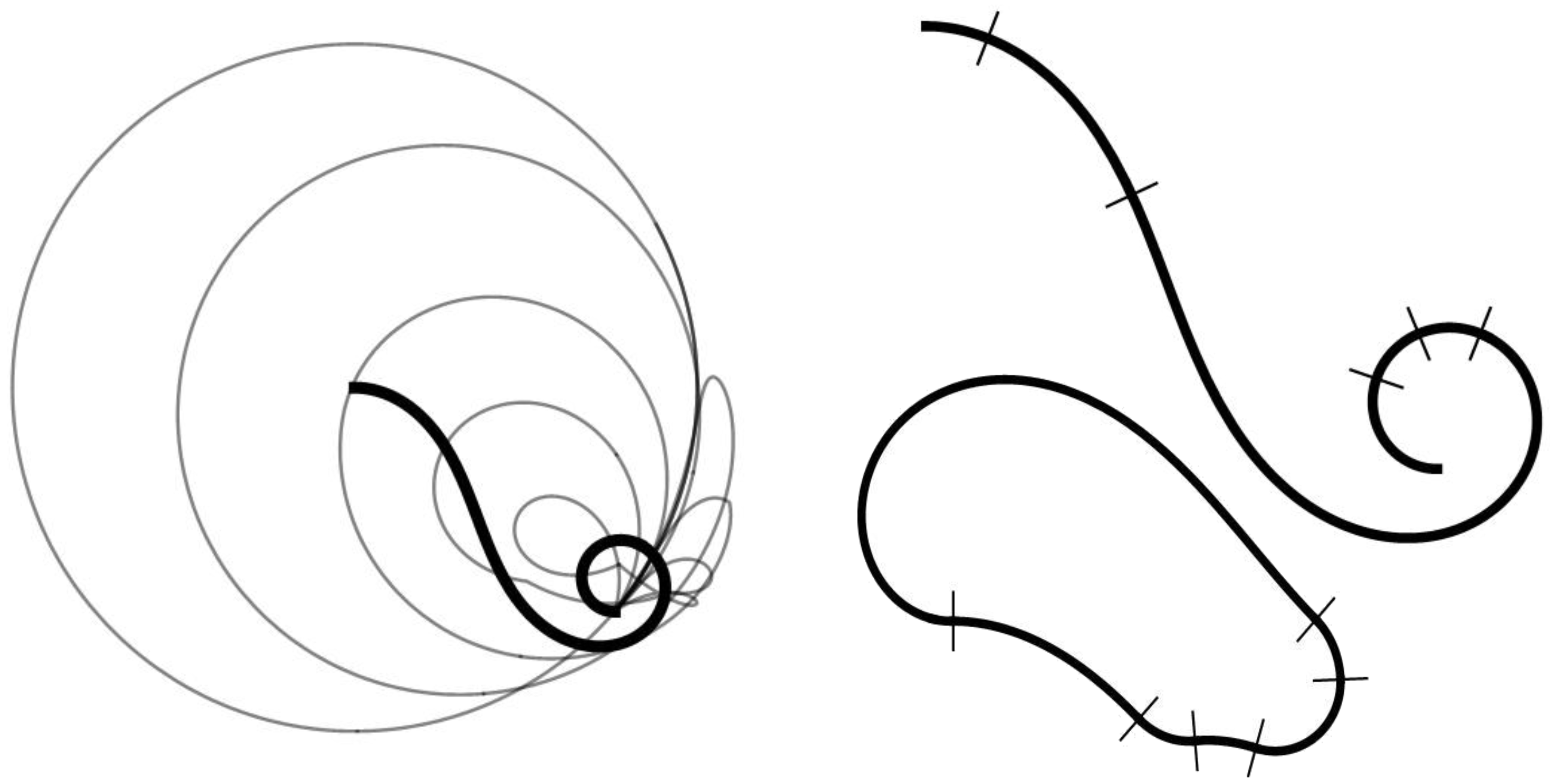}}};
    \node at (7.1,5.35) {$\gamma_1$};
      \node at (8.3,5.25) {$\gamma_2$};
      \node at (11.5,1.85) {$\gamma_3$};
      \node at (10.9,3.65) {$\gamma_4$};
      \node at (10.25,3.4) {$\gamma_5$};
      \node at (10.15,2.45) {$\gamma_6$};
      
      \node at (9.2,0.55) {$\gamma_1$};
      \node at (7.8,0.75) {$\gamma_2$};
      \node at (8.7,0.1) {$\gamma_5$};
      \node at (10.2,0.3) {$\gamma_6$};
      \node at (10.28,1.05) {$\gamma_4$};
      \node at (7.5,3.3) {$\gamma_3$};
      
      \node at (3.2,4) {$e_\sigma(I[h,t])$};
\end{tikzpicture}
\caption{Proper rearrangement of a curve split into $6$ arcs. The topological argument is the same as in \S \ref{main}: $e_\sigma(I[h,t])$ defines a family of loops starting as a circle centered in $(0,0)$ and contracting to $\gamma(1) \neq (0,0)$, which implies the starting point of the curve is contained in one of the intermediate loops of this family.} \label{inflationex}
\end{figure}
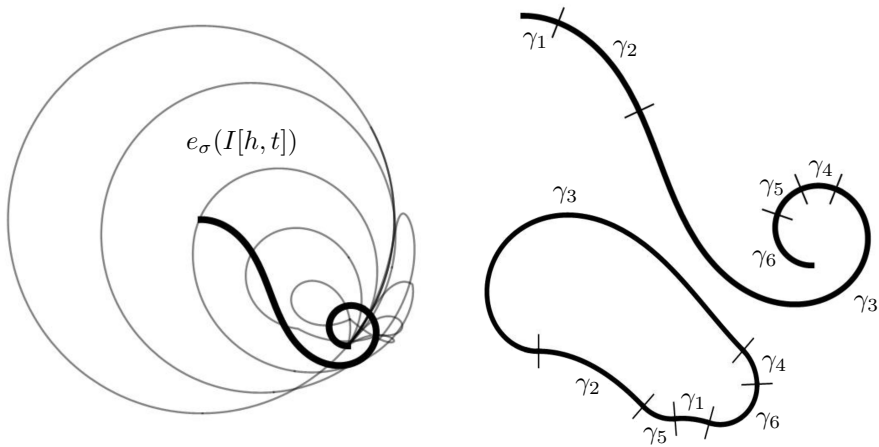

Figure \ref{inflationex} shows how a curve split into $6$ arcs can be properly rearranged. One can be a bit more precise about the maximum magnitude of the inflation factor $\delta$ exploiting uniform continuity of $e_\sigma$ on $D_k$. Nevertheless, we preferred the slightly less informative but leaner proof we gave above. 

The proof of Theorem \ref{char}, in the way we made sure that the perturbation of our starting cuts would not undermine the features of the loop $e_\sigma(I[0,t])$, presents one further analogy with the topological proof of the fundamental theorem of algebra. If $\bar{p}(z)=z^n$ then it is apparent that the loop $\bar{p}(\rho e^{it})$ for $\rho>0$ winds around the origin $n$ times, which is a property of little use since it is obvious where the roots of such a polynomial are. If $\bar{p}$ is changed to $p(z)=z^n+a_{n-1}z^{n-1}+...+a_1z+a_0$ though, the property about the winding number remains unalterd if $\rho$ is chosen large enough to make negligible the contribution of the terms of lower degree and therefore the contraction argument can still be used. The idea of exploiting the robustness of a topological argument, proven to work for a degenerate case, is also reminescent of the proof of the converse of the so called $4$-vertex theorem by Gluck \cite{gluck}, generalized on the same line some 30 years later by Dahlberg \cite{dahlberg}. A nice survey on this theorem and its proof(s) can be found in \cite{converse4}, where also the above observation about the fundamental theorem of algebra is pointed out.

\section{Some comments on higher dimensions and future work.} \label{highersec} 

A viable technique for showing $(k,0)$-rearrangeability in $\mathbb{R}^n$ would be to proceed in a manner analogous to the 2D case by assessing the contractibility in $\mathbb{R}^n \setminus \{ O \}$ of the image of $\partial D_k$ through $e_\sigma$.
If $e_\sigma (\partial D_k)$ contains the starting point $O$ then we have already achieved $(k-1,0)$-rearrangeability. Otherwise, up to homeomorphism a map $\partial D_k \cong S^{k-2} \rightarrow \mathbb{R}^n \setminus \{ O \}$ is induced. If this map is not contractible (it represents a non-trivial element in the $(k-2)$-th homotopy group of $\mathbb{R}^n \setminus \{ O \} \simeq S^{n-1}$) then there exist cuts in the interior of $D$ that provides $(k,0)$-rearrangeability. Nevertheless, in this setting the current lack of nice criteria to detect contractibility of the map induced by a certain permutation makes it hard to translate such considerations into explicit statements about rearrangeability. 

We conclude with pointing out a few possibilities for future work. A full characterization of planar curves that are $(k,0)$-rearrangeable or $(k,1)$-rearrangeable is an obvious next step. For higher dimensions it would be relevant to better understand the combinatorics of the image of $e_\sigma (\partial D_k)$ and to develop at least some neat sufficient conditions which guarantee that a curve is $(k,j)$-rearrangeable. $C^2$ curves in $\mathbb{R}^3$ whose curvature is constant are a promising family to study in this direction; in such class, if the curve is framed with the Frenet-Serret frame, the rearrangement procedure described in \S \ref{more} would provide $C^2$ regularity also at junction points.

\section{Acknowledgments.} The author acknowledges the support of the Austrian Science Fund (FWF): W1230, ``Doctoral Program Discrete Mathematics'' and of SFB-Transregio 109 ``Discretization in Geometry \& Dynamics'' funded by DFG and FWF (I2978). \\ The author would also like to thank Thilo R{\"o}rig and Johannes Wallner for helpful discussions and the Institut f{\"u}r Mathematik at TU Berlin for kind hospitality at the time this work was begun. 

\bibliography{bibliography} 
\bibliographystyle{amsplain}

\end{document}